\documentclass[11pt,reqno]{amsart}

\usepackage[T1]{fontenc}
\usepackage[centertags]{amsmath}
\usepackage{amsfonts}
\usepackage{amssymb}
\usepackage{amsthm}
\usepackage{newlfont}
\usepackage{fancyhdr,fancyvrb}
\usepackage{graphicx}
\usepackage{tikz}
\usepackage{pinlabel}
\usepackage{nextpage}
\usepackage{color}
\usepackage{hyperref}
\usepackage{esint}
\usepackage{soul}
\usepackage{longtable}      

\usepackage[scale=0.775]{geometry}

\theoremstyle{plain}
\newtheorem{theorem}{Theorem}[section]
\newtheorem{lemma}[theorem]{Lemma}

\theoremstyle{definition}
\newtheorem{definition}[theorem]{Definition}
\newtheorem{remark}[theorem]{Remark}
\newtheorem{example}[theorem]{Example}

\usepackage{tikz}             
\usetikzlibrary{decorations.pathreplacing}  
\usetikzlibrary{arrows}
\usetikzlibrary{calc}
\usetikzlibrary{through}
\usetikzlibrary{backgrounds}
\usetikzlibrary{patterns}
\usetikzlibrary{decorations.markings}

\tikzset{->-/.style={decoration={
  markings,
  mark=at position .5 with {\arrow{>}}},postaction={decorate}}}

\def\N{\mathbb{N}}
\def\R{\mathbb{R}}

\newcommand{\floor}[1]{\left\lfloor #1 \right\rfloor}
\newcommand{\ceil}[1]{\left\lceil #1 \right\rceil}
\newcommand{\norm}[2]{ \| #1 \|_{ #2 } }
\newcommand{\haa}[1]{( #1 )}
\newcommand{\abs}[1]{| #1 |}
\newcommand{\acc}[1]{ \{ #1 \} }
\newcommand{\accv}[2]{ \{ #1 \, | \, #2 \} }

\newcommand{\bighaa}[1]{\big( #1 \big)}
\newcommand{\bigbhaa}[1]{\big[ #1 \big]}

\newcommand{\bigacc}[1]{\big\{ #1 \big\}}
\newcommand{\bignorm}[2]{\big\| #1 \big\|_{ #2 } }
\newcommand{\bigabs}[1]{\big| #1 \big|}
\newcommand{\bigaccv}[2]{\big\{ #1 \, \big| \, #2 \big\}}

\newcommand{\Bighaa}[1]{\Big( #1 \Big)}
\newcommand{\Bigabs}[1]{\Big| #1 \Big|}

\newcommand{\Bigaccv}[2]{\Big\{ #1 \, \Big| \, #2 \Big\}}

\newcommand{\biggaccv}[2]{\bigg\{ #1 \, \bigg| \, #2 \bigg\}}

\newcommand{\bigghaa}[1]{\bigg( #1 \bigg)}

\newcommand{\lrhaa}[1]{\left( #1 \right)}
\newcommand{\lracc}[1]{\left\{ #1 \right\}}
\newcommand{\lrabs}[1]{\left| #1 \right|}

\newcommand{\lraccvl}[2]{\left\{ \left. #1 \, \right| \, #2 \right\}}

\newcommand{\lrbhaa}[1]{\left[ #1 \right]}

\newcommand{\ga}[2]{\frac{d #1 }{d #2 }}

\newcommand{\intabx}[4]{ \int_{ #1 }^{ #2 } { #3 } \, d { #4 } }

\newcommand{\rar}{\hspace{1mm} \Rightarrow \hspace{1mm}}

\newcommand{\argmin}{\operatorname*{arg\,min}}
\newcommand{\argmax}{\operatorname*{arg\,max}}

\newcommand{\markops}[1]{\operatorname{span} #1 }

\newcommand{\xto}[1]{\xrightarrow{ #1 }}

\newcommand{\dist}{\operatorname{dist}}

\newcommand{\mygraphic}[1]{\includegraphics[height=#1]{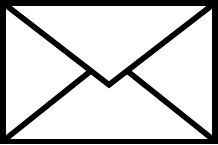}}
\newcommand{\myenv}{(\raisebox{0pt}{\mygraphic{.6em}})}

\def\XXint#1#2#3{{\setbox0=\hbox{$#1{#2#3}{\int}$}
     \vcenter{\hbox{$#2#3$}}\kern-.5\wd0}}

\newcommand{\Dom}{\operatorname{Dom}}

\newcommand{\parallelsum}{\, \mathbin{\!/\mkern-5mu/\!} \,}

\newcommand{\cE}{\mathcal E}
\newcommand{\cF}{\mathcal F}
\newcommand{\cG}{\mathcal G}

\newcommand{\cN}{\mathcal N}

\newcommand{\cR}{\mathcal R}

\newcommand{\cX}{\mathcal X}
\newcommand{\bzero}{\mathbf 0}
\newcommand{\ba}{\mathbf a}
\newcommand{\bb}{\mathbf b}

\newcommand{\bepsilon}{\boldsymbol \varepsilon}
\newcommand{\be}{\mathbf e}

\newcommand{\bg}{\mathbf g}
\newcommand{\bG}{\mathbf G}

\newcommand{\bOmega}{\boldsymbol \Omega}

\newcommand{\bu}{\mathbf u}
\newcommand{\bxi}{\boldsymbol \xi}
\newcommand{\bx}{\mathbf x}
\newcommand{\bX}{\mathbf X}
\newcommand{\by}{\mathbf y}
\newcommand{\bY}{\mathbf Y}
\newcommand{\bz}{\mathbf z}
\newcommand{\bZ}{\mathbf Z}

\DeclareFontFamily{U}{mathx}{\hyphenchar\font45}
\DeclareFontShape{U}{mathx}{m}{n}{
      <5> <6> <7> <8> <9> <10>
      <10.95> <12> <14.4> <17.28> <20.74> <24.88>
      mathx10
      }{}
\DeclareSymbolFont{mathx}{U}{mathx}{m}{n}
\DeclareFontSubstitution{U}{mathx}{m}{n}
\DeclareMathAccent{\widecheck}{0}{mathx}{"71}

\usepackage[scale=0.775]{geometry}

\numberwithin{equation}{section}

\title[Dynamics of screw dislocations]{Dynamics of screw dislocations: a generalised minimising-movements scheme approach}%
\author{Giovanni A.\@ Bonaschi}
\author{Patrick van Meurs}
\address{Centre for Analysis, Scientific computing and Applications (CASA), Department of Mathematics and Computer Science, Eindhoven University of Technology, P.O.~Box 513, 5600 MB Eindhoven, The Netherlands}
\email[G.~A.~ Bonaschi]{g.a.bonaschi@tue.nl}
\email[P.~van Meurs]{p.j.p.v.meurs@tue.nl}
\author{Marco Morandotti}
\address{SISSA -- International School for Advanced Studies, Via Bonomea, 265, 34136 Trieste, Italy. \emph{Tel:} +39 040 3787 422}
\email[M.~Morandotti \myenv]{marco.morandotti@sissa.it}

\date{\today. Preprint SISSA 38/2015/MATE}

\makeindex%

\begin{document}

\begin{abstract}

The gradient flow structure of the model introduced in \cite{CermelliGurtin99} for the dynamics of screw dislocations 
is investigated by means of a generalised minimising-movements scheme approach. The assumption of a finite number of available glide directions, together with the ``maximal dissipation criterion'' that governs the equations of motion, results into 
solving a differential inclusion rather than an ODE. This paper addresses how the model in \cite{CermelliGurtin99} is connected to a time-discrete evolution scheme which explicitly confines dislocations to move each time step along a single glide direction. It is proved that the time-continuous model in \cite{CermelliGurtin99} is the limit of these time-discrete minimising-movement schemes when the time step converges to $0$. The study presented here is a first step towards a generalization of the setting in \cite[Chap.~2 and 3]{AGS08} that allows for dissipations which cannot be described by a metric.
\end{abstract}
\maketitle%
{\small


\bigskip
}\bigskip


\section{Introduction}
\label{chap:screw:dyncs:sec:intro}

This paper is devoted to the application of the minimising-movement framework to the model introduced in \cite{CermelliGurtin99} for the evolution of a finite number of screw dislocations, under the constraint that dislocation movement only occurs along a finite set of \emph{glide directions}.
This constraint enforces that defects only move along the so-called \emph{slip planes}, which are determined by the crystallographic structure. 
The direction followed by each dislocations is the one dictated by a \emph{maximal dissipation criterion}: it is the glide direction which is closest in direction to the Peach-K\"ohler force that acts on the dislocation.

Yet, this confinement of the dislocation motion along glide directions is not captured by the model in \cite{CermelliGurtin99}. Indeed, this model allows for dislocation glide along more complex paths: dislocations can switch from one direction to another one, or even move along curved lines. 
These behaviours are called \emph{cross-slip} and \emph{fine cross-slip}, respectively, in \cite{CermelliGurtin99}.
Roughly speaking, at the turning point in cross-slip the dislocation switches direction from the most dissipative one to another one which becomes equally dissipative.
This switch between one glide direction and another happens at a much faster time scale in fine cross-slip, giving the dislocation a curved trajectory.

The aim in this paper is to understand how such curved trajectories emerge from dislocation motion along glide directions. Our main question is:
\smallskip
\begin{quote}
\centering
\emph{How does the evolution model in \cite{CermelliGurtin99} connect to a different model which confines dislocations to move along glide directions only?}
\end{quote}
\medskip
Our main contribution is to solve this question by proposing a time-discrete variational model (i.e.~a minimising-movement scheme) for the movement of dislocations, which confines the motion of any dislocation at each time step to a single glide direction, and by proving that the limiting equation, as the time step goes to zero, is given by the model in \cite{CermelliGurtin99}. Interestingly, the confinement of dislocations to move along glide directions leads to a dissipation which cannot be described as the square of a distance (it will not satisfy the triangle inequality). Therefore, the well-known convergence of minimising-movement schemes (see e.g.~\cite{AGS08}) does not apply, and our result is a first step in generalizing the convergence of minimising-movement schemes.

After a brief prologue on dislocations in Section \ref{chap:screw:dyncs:ssec:dlcns}, we introduce the model from \cite{CermelliGurtin99} in Section~\ref{chap:screw:dyncs:sec:intro}, and our modified setting in Section~\ref{chap:screw:dyncs:ssec:intro:MMS}. We then provide an intuitive example (see Example~\ref{ex:screw:dlc}) to illustrate the dynamics of screw dislocations. In Section~\ref{sec:dislo_res_com} we describe and discuss Theorem \ref{thm:chap:screw:dyncs}, which is our main result that describes the connection between our modified setting and \cite{CermelliGurtin99}.

\subsection{Screw dislocations} \label{chap:screw:dyncs:ssec:dlcns}

\emph{Dislocations} are line-defects on the atomic length-scale in the crystallographic structure of the material. Typical metals contain many dislocations (for example, cold-rolled metal has a dislocation density of $10^{15}$ m$^{−2}$ \cite[p.~20]{HullBacon01}, which translates into $1000$ km of dislocation line per cubic millimetre). Their collective motion results in \emph{plastic deformation} of the material on length-scales between $1$ $\mu$m and $1$ mm. Figure~\ref{figure:edgeescrew} shows two straight segments of a dislocation line in a cubic crystallographic lattice. The result of dislocations being line-defects is that they induce a stress in the material. Dislocations elsewhere in the material may move as a consequence of the stress. This results in an intriguing and complex system of interaction line-defects. For more details we refer to \cite{HirthLothe82,HullBacon01}.

%
\begin{figure}[htb]
\centering
\includegraphics[scale=0.4]{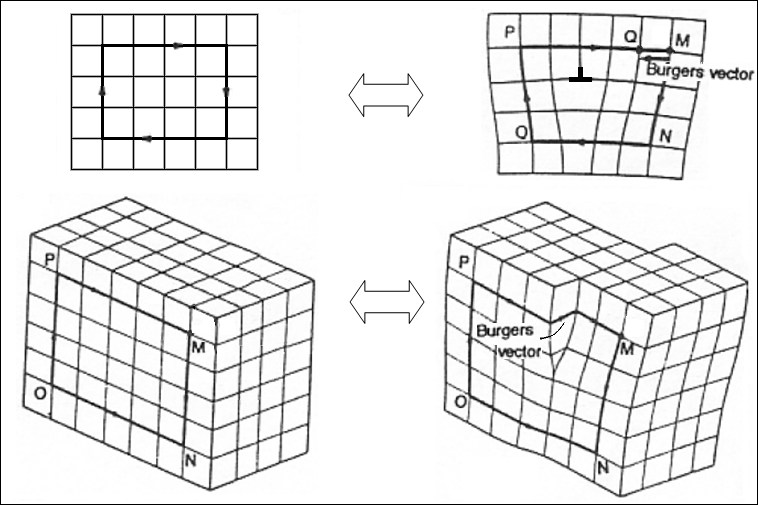}
\caption{Pictorial representation of edge (above) and screw (below) dislocations. (Source: \texttt{https://commons.wikimedia.org/wiki/File:Vector\_de\_Burgers.PNG}. Image by David Gabriel Garc\'ia Andrade.)}
\label{figure:edgeescrew}
\end{figure}

We study the evolution of a set of $n$ straight and parallel dislocation lines of 'screw'-type (see Figure \ref{figure:edgeescrew}) in a three-dimensional elastic material undergoing \emph{antiplane shear}. The body is modelled by an infinite cylinder $\Omega\times\R$. Therefore, we characterise the positions of the screw-dislocations by points in the cross-section $\Omega\subset\R^2$. Following the model proposed in \cite{CermelliGurtin99}, we assume that the $n$ screw dislocations $\bz_1,\ldots,\bz_n\in\Omega$ are constrained to move along 
%
\emph{glide directions}. The glide directions are determined by those directions along which the atoms in the material are most densely packed. For instance, for a cubic atomic lattice, there are four glide directions given by $\acc{\pm \be_1, \pm \be_2} \subset \mathbb S^1$ ($\be_i$ being the standard basis vectors in $\R^2$), and for a body-centred cubic or face-centred cubic lattice, there are six glide directions which span the triangular grid of equilateral triangles. We consider a more general setting in which the set of $N \in 2 \N{}$ glide directions is given by
\begin{equation} \label{for:defn:cG}
\cG := \acc{ \bg_1, \ldots, \bg_N } \subset \mathbb S^1,
\end{equation}
which satisfies the basic properties
\begin{equation} \label{for:G:props}
  \bg \in \cG \rar -\bg \in \cG,
  \quad \text{and} \quad
  \markops{\cG} = \R^2. 
\end{equation} 

Under 
the maximal dissipation criterion described above, the velocity vector of any smooth solution $t\mapsto \bZ(t) = ( \bz_1, \ldots, \bz_n) (t)$ is aligned along one of the glide directions at each time $t$. It turns out that generically, solutions of this type may not exist; this phenomenon is illustrated in Example \ref{ex:screw:dlc} below, and the simulations in \cite[Sec.~4]{BlassFonsecaLeoniMorandotti15} confirm this observation. In order to obtain existence of solutions, one is forced to allow the motion of dislocations along a direction different than any of the glide directions, called \emph{fine cross-slip}~\cite{CermelliGurtin99} (or `sliding motion' as in the theory of discontinuous differential equations (see e.g.~\cite{BernardoBuddChampneysKowalczyk08})). 
Note that fine cross-slip seemingly contradicts the model assumption that screw dislocations only move along glide directions. 
The variational evolution model that we propose here shows how fine cross-slip can be incorporated in the equations of motion, as the limit of rapidly alternating directions.

\subsection{The model in \cite{CermelliGurtin99}}
Before defining the evolution model which confines dislocations to move along glide directions only, we introduce the time-continuous model in \cite{CermelliGurtin99} in more detail. Based on \cite[(1.10)]{CermelliGurtin99}, in \cite{BlassFonsecaLeoniMorandotti15} the following differential inclusion is posed to model the time-continuous dynamics of screw dislocations:
\begin{equation} \label{for:evo:screw:BFLM}
  \ga {\bz_i}t (t)
  \in F_i (\bZ (t)), 
  \qquad \text{for all } t \geq 0, \: i = 1,\ldots,n,
\end{equation}
where $\bZ = (\bz_1, \ldots, \bz_n)^T \in \Omega^n$ denotes the positions of $n$ screw dislocations in a prescribed Lipschitz domain $\Omega \subset \R^2$. Each dislocation is associated with a \emph{Burgers vector} $\bb_i$, which describes the direction of lattice mismatch due to the presence of the dislocation itself (see Figure \ref{figure:edgeescrew}). Due to the assumption of antiplane shear, in our case the Burgers vectors are oriented along the $\be_3$-direction, so that they can be described by a scalar quantity $b_i$, which we call the \emph{Burgers modulus}: $\bb_i=b_i\be_3$. The quantities $b_i$ can be positive ($b_i = 1$) or negative ($b_i = -1$) and their sign is responsible for the attraction or repulsion of dislocations, as it can be seen in \eqref{for:defn:cE:chap:screw} below.


The set-valued function $F_i (\bZ)$ 
projects the force $f_i (\bZ)$ on the $i$-th dislocation onto the nearest glide direction. More precisely, we define the (nonlinear) multi-valued projection operator $P_\cG : \R^2 \to \R^2$ by
\begin{equation*}
  P_\cG \xi := \Bigaccv{ (\bg \cdot \xi ) \bg }{ \bg \in \argmax_{\tilde \bg \in \cG} \tilde \bg \cdot \xi }.
\end{equation*}
Then, $F_i (\bZ) = \operatorname{co} P_\cG f_i (\bZ)$, where $\operatorname{co}$ is the convex hull. The force $f_i (\bZ) = - \nabla_{\bz_i} \cE (\bZ)$ is defined in \cite[(4.4)]{BlassMorandotti14prep}, in which $\cE$ takes the form
\begin{equation} \label{for:defn:cE:chap:screw}
  \cE (\bZ) := \varphi (\bZ) + \sum_{i = 1}^n \sum_{\substack{ j = 1 \\ j\neq i } }^n - b_i b_j \log |\bz_i - \bz_j|,
\end{equation}
where $\varphi \in C^\infty (\Omega^n)$ (see \cite[Lemma~5.1]{BlassMorandotti14prep}) is bounded from above and satisfies $\varphi (\bZ) \to -\infty$ as $\dist (\bz_i, \partial \Omega) \to 0$ for any $i \in \N{}$. The logarithmic interaction potential corresponds to the Peach-K\"ohler force induced by screw dislocations, and $\varphi$ describes the effect of the traction-free boundary condition at $\partial \Omega$. For our purposes it is enough to have $\cE \in C^1 (\Omega^n \setminus S)$, where the set $S$ of singular points is given by
\begin{equation} \label{for:defn:S:chap:screw}
  \bZ \in S := \partial \Omega^n \cup \bigaccv{ \bZ \in \Omega^n }{ \exists \, i \neq j : \bz_i = \bz_j }.
\end{equation}

In \cite{BlassFonsecaLeoniMorandotti15}, local-in-time existence and uniqueness of solutions to \eqref{for:evo:screw:BFLM} is proven for suitable initial conditions. Here, a solution to \eqref{for:evo:screw:BFLM} is defined to be an absolutely continuous curve which satisfies \eqref{for:evo:screw:BFLM} for a.e.~$t > 0$. The proof of well-posedness in \cite{BlassFonsecaLeoniMorandotti15} relies on the theory for differential inclusions developed by Filippov \cite{Filippov88}.

Next we discuss the result in \cite{AlicandroDeLucaGarroniPonsiglione15prep} in the context of our main question above. In \cite{AlicandroDeLucaGarroniPonsiglione15prep} a fully atomistic model is considered to describe the energy for a given configuration $\bZ$ as a function of the atom spacing $\varepsilon$. The dynamics are defined by imposing a dissipation potential which is the square of a norm which is minimal in the glide directions. The main results concern the derivation of the effective energy in the limit $\varepsilon \to 0$ (which is also done in \cite{AlicandroDeLucaGarroniPonsiglione14}), and the passage to the limit in the related minimising-movement scheme as the time step converges to $0$. When the dissipation is chosen to be the square of the crystalline norm given by
\begin{equation} \label{for:defn:cryst:norm}
 \| \bx \| := \inf \bigaccv{ 
 \alpha + \beta 
 }{
  \text{there exist } \alpha, \beta \geq 0 \text{ and } \bg, \tilde \bg \in \cG \text{ such that } \bx = \alpha \bg + \beta \tilde \bg 
  },
\end{equation}
then the evolution \eqref{for:evo:screw:BFLM} is obtained as a generalised gradient flow (see, e.g., \cite{Mielke14LN} or Definition \ref{defn:genzd:GF}). While this result connects rigorously a detailed atomic description of the dynamics of screw dislocations to the time-continuous evolution \eqref{for:evo:screw:BFLM}, it does not answer our question above. Indeed, the dissipation in the atomic model is assumed to be derived from a norm, by which dislocations are not confined to move along glide directions.

\subsection{Evolution constrained along $\cG$} \label{chap:screw:dyncs:ssec:intro:MMS}

To confine screw dislocations to move along glide directions only, we impose a discrete-in-time evolution model as a minimising-movement scheme. The special feature of this scheme is that the related distance is replaced by a `quasi-distance' $D$ given by
\begin{align} \label{for:defn:d}
d &: \R^2 \times \R^2 \to [0, \infty],
&d (\bx, \by) 
&:= \begin{cases}
  | \bx - \by |, & \text{if } \bx - \by \in \mathbb R \mathcal G, \\
  \infty, & \text{otherwise},
\end{cases} \\\label{for:defn:D}
D &: (\R^2)^n \times (\R^2)^n \to [0, \infty],
&D^2 (\bX, \bY)
&:= \sum_{i=1}^n d^2 (\bx_i, \by_i).
\end{align}
Since $d$ is only finite along glide directions, $D$ violates the triangle inequality.

The minimising-movement scheme does not make sense for the energy $\cE$ defined in \eqref{for:defn:cE:chap:screw}, because $\cE$ is not bounded from below on the set $S$ of singular points \eqref{for:defn:S:chap:screw}. For this reason, we define a regularization of $\cE$ as follows: for any $\varepsilon > 0$, we take $\cE_\varepsilon \in C^\infty ((\R^2)^n)$ bounded from below, such that $\cE_\varepsilon = \cE$ on the closed set 
$$
\bOmega_\varepsilon := \bigaccv{ \bZ \in \Omega^n }{ \dist (\bZ, S) \geq \varepsilon }. 
$$

To define the minimising-movement scheme, we take $\tau > 0$ as the time step, $T > 0$ as the end time, and $\bZ^0 \in \Omega^n$ as the initial condition. A discrete-in-time solution $\bZ^k_\tau \in \Omega^n$ at the time points $t_k = k \tau$ is defined by 
\begin{equation} \label{for:MMS:D:intro}
  \left\{ \begin{aligned}
    \bZ_\tau^{k+1} &\in \argmin_{\bX \in (\R^2)^n} \Phi \big( \bZ_\tau^{k},\bX,\tau \big), &k = 0, \ldots, \ceil{ T/\tau } - 1, \\
    \bZ_\tau^0 &= \bZ^0,
  \end{aligned} \right.
\end{equation}
in which the functional $\Phi$ is given by
\begin{align*} 
  \Phi \big( \bX,\bY,\tau \big) 
      := \frac{ D^2 (\bX, \bY) }{2 \tau} + \cE_\varepsilon (\bY). 
\end{align*}
By the definition of $D$, dislocations are confined to move along one glide direction only for each time step.

\subsection{Illustration of screw dislocation motion}
\label{chap:screw:dyncs:sec:ex}

Example \ref{ex:screw:dlc} illustrates the evolution defined by \eqref{for:evo:screw:BFLM}. More involved examples concerning interacting dislocations are given in \cite{BlassFonsecaLeoniMorandotti15}.

\begin{example} \label{ex:screw:dlc}
This example is based on \cite[Figure 1]{CermelliGurtin99}. We consider a single screw dislocation in a medium $\Omega$ with glide directions given by $\cG = \acc{ \pm \be_1, \pm \be_2 }$ with $e_i$ the standard unit vectors. We impose a continuous force field $F : \R^2 \to (0, \infty)^2$. We choose $F$ such that the set $\accv{\bx \in \R^2}{F(\bx) \parallelsum (\be_1 + \be_2)}$ equals the boundary of a bounded domain $\cF$, in which $F \cdot \be_2 > F \cdot \be_1$, and outside of which $F \cdot \be_1 > F \cdot \be_2$. By \eqref{for:evo:screw:BFLM}, the dislocation will move along $\be_2$ if it is inside $\cF$, and along $\be_1$ if it is outside $\R^2$. We show trajectories of the dislocation for several initial conditions in Figure \ref{fig:screw:dlc:dyn}. We limit our attention here to the direction of the movement and not to its speed.

\begin{figure}[h]
\centering
\begin{tikzpicture}[scale=0.5, >= latex] 
\fill[color=gray] (-1.5,0.5) circle (0.1);
\draw[->-, thick, color=gray] (-1.5, 0.5) -- (0.76,0.5);
\draw[->-, thick, color=gray] (0.76,0.5) -- (0.8,4.7);
\draw[->-, thick, color=gray] (5,6) -- (12.1,6);
\draw[->-, thick, color=gray] (15,7) -- (22,7);

\fill[color=gray] (6.5,1.5) circle (0.1);
\draw[->-, thick, color=gray] (6.5, 1.5) -- (6.5,5.01);
\draw[->-, thick, color=gray] (6.49,5) -- (10.9,5);

\fill[color=gray] (6.8,1.2) circle (0.1);
\draw[->-, thick, color=gray] (6.8,1.2) -- (12.63,1.2);
\draw[->-, thick, color=gray] (12.62,1.19) -- (12.62,6.4);

\fill[color=gray] (19,2.5) circle (0.1);
\draw[->-, thick, color=gray] (19, 2.5) -- (19,6.13);
\draw[->-, thick, color=gray] (18.99,6.13) -- (22,6.13);

\fill[color=gray] (19.5,2) circle (0.1);
\draw[->-, thick, color=gray] (19.5, 2) -- (22,2);

\draw (0,3) .. controls (0,1.5) and (0.5,0) .. (2,0);
\draw[dotted] (2,0) .. controls (6,0) and (6,2) .. (10,2);
\draw (10,2) .. controls (12,2) and (12,1) .. (14,1);
\draw[dotted] (14,1) .. controls (17.5,1) and (21,2.5) .. (21,4);
\draw (21,4) .. controls (21,5.5) and (18,7) .. (15,7);
\draw[very thick] (15,7) .. controls (11.5,7) and (11.5,4) .. (8,4);
\draw (8,4) .. controls (6.5,4) and (6.5,6) .. (5,6);
\draw[very thick] (5,6) .. controls (2.5,6) and (0,5) .. (0,3);

\fill (5,6) node[anchor = south] {$A$} circle (0.1);
\fill (15,7) node[anchor = south] {$B$} circle (0.1);
\draw[->] (-2, 5) -- (-2,6.5) node[anchor=south] {$\be_2$};
\draw[->] (-2, 5) -- (-0.5,5) node[anchor=west] {$\be_1$};
\end{tikzpicture}
\caption{Typical setting of Example \ref{ex:screw:dlc}, along with a few trajectories of the screw dislocation. The closed curve depicts the points $\bx$ at which $F(\bx) \parallelsum (\be_1 + \be_2)$. }
\label{fig:screw:dlc:dyn}
\end{figure}
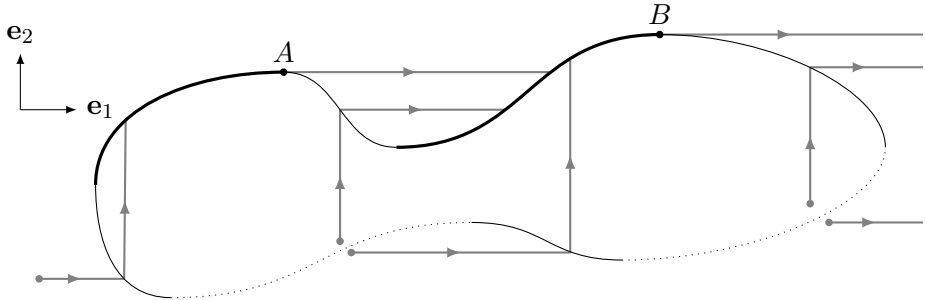

The interesting part of the dynamics is the behaviour of the dislocation at $\partial \cF$. On $\partial \cF$ the right-hand side of \eqref{for:evo:screw:BFLM} is multi-valued. $\partial \cF$ is called the \emph{ambiguity set}. We distinguish the following three types of ambiguity sets: sources (dotted line), cross-slip (thin line), and fine cross-slip (thick line). 

If the screw dislocation is at a cross-slip point (say at the left-lower part of Figure \ref{fig:screw:dlc:dyn}) at $t = t_1$, then it can move in any direction $\theta \be_1 + (1 - \theta)\be_2$ for $\theta \in [0,1]$. Whichever direction is chosen, at $t = t_1 + \Delta t$ the dislocation is inside $\cF$ for any $\Delta t > 0$ small enough, and hence it will move in direction $\be_2$ as soon as it enters $\cF$.

Following a similar reasoning, at a fine cross-slip point (say at the left-upper corner of Figure \ref{fig:screw:dlc:dyn}), we conclude that a dislocation tends to \emph{move along} a fine cross-slip set. If it hits the end point $A$, it will move along direction $\be_1$. 

Following a similar reasoning, we conclude that the initial-value problem may not have a unique solution at a source point. The word `source' should be understood here in the sense that time paths will never cross it. 

\end{example}

\subsection{Result and comments}
\label{sec:dislo_res_com}

Our main result is Theorem \ref{thm:chap:screw:dyncs}. It answers our main question by showing compatibility of the evolution constrained to glide directions with the evolution \eqref{for:evo:screw:BFLM} that can cause dislocations to move in any direction. In particular, it gives a precise meaning to fine cross-slip in terms of oscillating choices for the preferred glide direction in consecutive time steps as the time step size converges to $0$.

More precisely, Theorem \ref{thm:chap:screw:dyncs} states in what sense solutions to the minimising-movement scheme \eqref{for:MMS:D:intro} converge to a solution of 
\begin{equation} \label{for:evo:screw:BFLM:eps}
  \left\{ \begin{aligned}
    \ga {\bz_i}t (t)
  &\in \operatorname{co} P_\cG \bighaa{ - \nabla_{\bz_i} \cE_\varepsilon (\bZ (t)) }, 
  \quad \text{for all } t \in (0, T], \: i = 1,\ldots,n, \\
    \bz_i (0) &= \bz_i^0, 
  \end{aligned} \right.  
\end{equation}
as the time step $\tau \to 0$. By the definition of the regularization $\cE_\varepsilon$, any solution of \eqref{for:evo:screw:BFLM:eps} satisfies \eqref{for:evo:screw:BFLM} on the time interval $[0, T_\varepsilon]$, where $T_\varepsilon$ is chosen such that any solutions $\bZ (t)$ remains within the set $\bOmega_\varepsilon$. 
We notice that if we fix $\varepsilon_0>0$, then $T_{\varepsilon_0}>0$; moreover, it is easy to see that $T_\varepsilon$ is a non-increasing function of $\varepsilon$.
Therefore, $\varepsilon > 0$ can be chosen arbitrarily. 

As a corollary of Theorem \ref{thm:chap:screw:dyncs}, we obtain existence of solutions to \eqref{for:evo:screw:BFLM} up to the first time at which either two dislocations annihilate, or when a dislocation leaves the domain $\Omega$. This generalises the local-in-time well-posedness result of \cite{BlassFonsecaLeoniMorandotti15}. Theorem \ref{thm:chap:screw:dyncs} also reveals how \eqref{for:evo:screw:BFLM} can be written as a generalised gradient flow (see Section \ref{chap:screw:dyncs:ssec:MMSnEDI}). This gradient flow structure is also obtained in \cite[(3.11)]{AlicandroDeLucaGarroniPonsiglione15prep}.


Theorem \ref{thm:chap:screw:dyncs} is proved for all energies $E \in C^1 ((\R^2)^n) \cap W^{1,\infty} ((\R^2)^n)$. To the best of our knowledge, Theorem \ref{thm:chap:screw:dyncs} provides the first extension of the theory in \cite[Chapters~2 and 3]{AGS08} to dissipations which are not related to a distance. We wish to comment on further possible generalizations:
\begin{itemize}
\item Theorem \ref{thm:chap:screw:dyncs} extends to higher dimensions for the particle positions with little modifications to the proof. We provide the details in Section~\ref{ssec:extn}.
\item Our proof of Theorem \ref{thm:chap:screw:dyncs} heavily relies on the regularity of the energy $E$. In view of the discontinuities in the projection operator $P_\cG$, it seems reasonable to require the force field $- \nabla E$ to be continuous for \eqref{for:evo:screw:BFLM} to be well-posed. Hence, we do not aim to weaken the regularity conditions on $E$. 
\item By our regularisation $\cE_\varepsilon$ of $\cE$ in \eqref{for:defn:cE:chap:screw} we cannot describe the annihilation of dislocations or the absorption of dislocations at $\partial \Omega$. In Section~\ref{chap:screw:dyncs:sec:NnC} we describe an open problem concerning the description of these effects in a variational framework.
\end{itemize}

The remainder of this paper is structured as follows. In Section \ref{chap:screw:dyncs:ssec:MMSnEDI} we define precisely the minimising-movement scheme (MMS) and describe the evolution model \eqref{for:evo:screw:BFLM} in a variational framework by means of an energy dissipation inequality (EDI). Section \ref{chap:screw:dyncs:ssec:prelim} presents \emph{a priori} estimates which are required for the proof of Theorem \ref{thm:chap:screw:dyncs}, which is contained in  Section~\ref{chap:screw:dyncs:sec:thm:and:pf}. Section~\ref{sec:extn} concerns the extension and the open problem related to Theorem \ref{thm:chap:screw:dyncs} as mentioned above.

\section{Preliminaries}
\label{chap:screw:dyncs:sec:prelim}

\subsection{Notation}
\label{chap:screw:dyncs:ssec:not}

Here we list symbols and abbreviations that we use throughout this paper: 
\begin{longtable}{lll}
$|\bx|$ & Euclidean norm & \\
$\| \cdot \|$ & crystalline norm in $\R^2$ with respect to $\cG$ & \eqref{for:defn:cryst:norm} \\
$\| \cdot \|_\ast$ & dual norm of $\| \cdot \|$ & \eqref{for:defn:dual:cryst:norm} \\
$\bx \parallelsum \by$ & $\bx$ is parallel to $\by$, i.e.~$\bx / |\bx| = \by / |\by|$ & \\ 
$\operatorname{co}$ & convex hull & \\
$\hat d$ & metric induced by $\|\cdot\|$ & \eqref{for:defn:Dhat} \\
$d$ & `quasi-distance', which is only finite along $\bg \in \cG$ & \eqref{for:defn:D} \\
$\hat D$ & extension of the norm $\hat d$ to $(\R^2)^n$ & \eqref{for:defn:Dhat} \\
$D$ & extension of the `quasi-distance' $d$ to $(\R^2)^n$ & \eqref{for:defn:D} \\
$\cE$ & specific energy functional for edge dislocations; $\cE : \Omega^n \to \overline \R{}$ & \eqref{for:defn:cE:chap:screw} \\
$E$ & generic energy functional; $E \in C^1 ((\R^2)^n) \cap W^{1,\infty} ((\R^2)^n)$ & \\
EDI & energy-dissipation inequality & \eqref{for:EDIhat} \\
$\Phi$ & functional to be minimised in the $D$-MMS & \eqref{for:defn:Phi:E} \\
$\bg$, $\cG$ & glide direction in $\mathbb S^1$, and set of all glide directions & \eqref{for:G:props} \\
$\Lambda_\bg$ & cone around $\bg$ & \eqref{for:defn:Lambda} \\ 
MMS & minimising-movement scheme & \eqref{for:MMS:D} \\
$\bZ$ & particle positions; $\bZ = (\bz_1, \ldots, \bz_n)^T \in (\R^2)^n$ & \\
$\bZ_\tau^k$ & solution to the $D$-MMS at the $k$th time step & \eqref{for:MMS:D} \\
$\overline \bZ_\tau$ & step function related to $(\bZ_\tau^k)_k$; $\overline \bZ_\tau : [0, T] \to (\R^2)^n$ & \eqref{for:defn:Zbar} \\
$\bZ_\tau^\Gamma$ & De Giorgi interpolant; $\bZ_\tau^\Gamma : [0, T] \to (\R^2)^n$ & \eqref{for:defn:De:Giorgi:itplt} \\
\end{longtable}

\subsection{Glide directions}

In this section we show several basic properties of the set of glide directions $\cG$ defined in \eqref{for:defn:cG} satisfying the basic properties \eqref{for:G:props}. Figure \ref{fig:cG} illustrates an example. As a consequence of \eqref{for:G:props}, the number of glide directions $N \geq 4$ is even. As $\mathbb S^1$ has a cyclic ordering, we assume $\bg_1, \ldots, \bg_N$ to be ordered counter-clockwise, and we set for convenience $\bg_0 := \bg_N$ and $\bg_{N+1} := \bg_1$. Then, we define the bisectors
\begin{equation*}
  \bg_i' 
  := \frac{ \bg_i + \bg_{i+1} }{ |\bg_i + \bg_{i+1}| } 
  \in \mathbb S^1,
  \qquad i = 1,\ldots,N.
\end{equation*}
We define $\Lambda_{\bg_i}$ as the cone spanned by the bisectors $\bg_{i-1}', \bg_i'$ surrounding $\bg_i$, i.e.
\begin{equation} \label{for:defn:Lambda}
  \Lambda_{\bg_i} 
  := 
  \bigaccv{ 
\alpha \bg_{i-1}' + \beta \bg_i'
}{
\alpha, \beta > 0
}
  \subset \R^2.
\end{equation}
As a consequence,
\begin{equation} \label{for:Lambda:prop}
  \Lambda_\bg 
  = 
  \bigaccv{ 
\bx \in \R^{2}
}{
\bg \text{ is the unique maximizer of } \tilde \bg \cdot \bx \text{ for all } \tilde \bg \in \cG
}.
\end{equation}

The set $\cG$ induces the crystalline norm $\|\cdot \|$ given by \eqref{for:defn:cryst:norm}. The unit ball of $\| \cdot \|$ is given by $\operatorname{co} \cG$. The dual norm reads
\begin{equation} \label{for:defn:dual:cryst:norm}
 \| \bx \|_\ast := \max_{\bg \in \mathcal G} \bg \cdot \bx.
\end{equation}
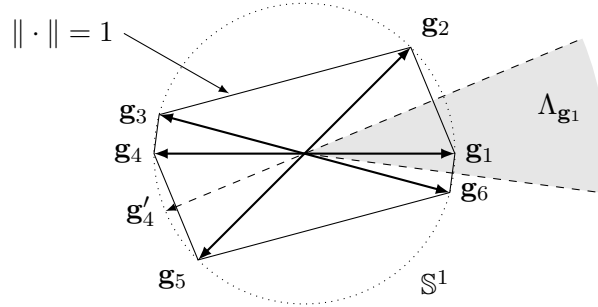
\begin{figure}[h]
\centering
\begin{tikzpicture}[scale=2, >= latex] 
\fill[fill=gray!20!white] (0,0) -- (1.9829,-0.2611) arc (-7.5:22.5:2) -- cycle;
\draw[dotted] (0,0) circle (1);
\draw (0.7071,-0.7071) node[anchor = north west]{$\mathbb S^1$};

\draw[thick, <->] (-1,0) node[anchor = east]{$\bg_4$} -- (1,0) node[anchor = west]{$\bg_1$};
\draw[thick, <->] (-0.7071,-0.7071) node[anchor = north east]{$\bg_5$} -- (0.7071,0.7071) node[anchor = south west]{$\bg_2$};
\draw[thick, <->] (-0.9659,0.2588) node[anchor = east]{$\bg_3$} -- (0.9659,-0.2588) node[anchor = west]{$\bg_6$};
\draw[dashed, <-] (-0.9239,-0.3827) node[anchor = east]{$\bg_4'$} -- (1.8478,0.7654); 
\draw[dashed] (0,0) -- (1.9829,-0.2611);

\draw [<-] (-0.5,0.4) -- (-1.2,0.8) node[anchor = east]{$\|\cdot\| = 1$};
\draw (1.7, 0.3) node{$\Lambda_{\bg_1}$};
\draw (1,0) -- (0.7071,0.7071) -- (-0.9659,0.2588) -- (-1,0) -- (-0.7071,-0.7071) -- (0.9659,-0.2588) -- (1,0); 
\end{tikzpicture}
\caption{Example of a set glide directions $\cG$ satisfying \eqref{for:G:props}. The figure also depicts several related objects such as: the bisector $\bg_4'$, the cone $\Lambda_{\bg_1}$, and the unit sphere of the crystalline norm $\|\cdot\|$.}
\label{fig:cG}
\end{figure}

\subsection{MMS and EDI}
\label{chap:screw:dyncs:ssec:MMSnEDI}

We consider $(\R^2)^n$ as the state space equipped with the `quasi-distance' $D$ defined in \eqref{for:defn:D}, and some $E \in C^1 ((\R^2)^n) \cap W^{1,\infty} ((\R^2)^n)$ as the energy functional. We note that this class of energy functionals includes the energy $\cE_\varepsilon$ introduced in Section \ref{chap:screw:dyncs:sec:intro} as a regularization of the energy $\cE$ \eqref{for:defn:cE:chap:screw} which describes the system of screw dislocations. \smallskip\\
\textbf{MMS}: For a time step $\tau > 0$, an initial condition $\bZ^0 \in (\R^2)^n$, and a set $\cG$ of $N$ glide directions satisfying \eqref{for:G:props}, we consider the $D$-MMS given by 
\begin{equation} \label{for:MMS:D}
  \left\{ \begin{aligned}
    \bZ_\tau^{k+1} &\in \argmin_{\bX \in (\R^2)^n} \Phi \big( \bZ_\tau^{k},\bX,\tau \big), &k = 0, \ldots, \ceil{ T/\tau } - 1, \\
    \bZ_\tau^0 &= \bZ^0,
  \end{aligned} \right.
\end{equation}
in which the functional $\Phi$ is given by
\begin{align} \label{for:defn:Phi:E}
  \Phi : (\R^2)^n \times (\R^2)^n \times (0, \infty) \to (-\infty, \infty],
  \qquad
      \Phi \big( \bX,\bY,\tau \big) 
      := \frac{ D^2 (\bX, \bY) }{2 \tau} + E (\bY). 
\end{align}
The only difference with \eqref{for:MMS:D:intro} is that we use the more general energy functional $E$. Existence of a solution $(\bZ_\tau^{k})_k$ to \eqref{for:MMS:D} is guaranteed by $E$ being bounded from below and lower-semicontinuous. Uniqueness does not hold in general. 

For any solution $(\bZ_\tau^{k})_k$ to \eqref{for:MMS:D} we define two interpolation curves as mappings from $[0, T]$ to $(\R^2)^n$. The first one is the piecewise constant interpolant
\begin{equation} \label{for:defn:Zbar}
    \overline \bZ_\tau (t) := \bZ_\tau^{k} \quad
    \text{for } t \in \big( (k-1) \tau, k\tau \big].
\end{equation} 
The second one is a De Giorgi interpolant $\bZ^\Gamma_\tau (t)$, which is defined as follows. For any $t \in [0,T]$, let 
  \begin{equation} \label{for:def:itau}
    k_\tau := \lceil t/\tau \rceil - 1,
    \quad \text{and note that}
    \quad t \in ( \tau k_\tau, \tau (k_\tau + 1) ].
  \end{equation}  
  Let $\delta := t - \tau k_\tau \in (0, \tau]$, and finally
  \begin{equation} \label{for:defn:De:Giorgi:itplt}
    \bZ^\Gamma_\tau (t) \in \argmin_{\bX \in (\R^2)^n} \Phi(\bZ_\tau^{k_\tau} , \bX, \delta).
  \end{equation}
As $\delta \leq \tau$, the following basic property (see the proof of \cite[Lemma~3.1.2]{AGS08}) holds
  \begin{equation} \label{for:est:on:ZGamma}
    D \big( \bZ_\tau^{k_\tau}, \bZ^\Gamma_\tau (t) \big)
    \leq D \big( \bZ_\tau^{k_\tau}, \overline \bZ_\tau (t) \big)
    = D \big( \bZ_\tau^{k_\tau}, \bZ_\tau^{k_\tau + 1} \big).
  \end{equation}

The main challenge in passing to the limit $\tau \to 0$ in \eqref{for:defn:Zbar} is that $D$ is not a metric. If it were a metric, the techniques in \cite[Chapter~2 and 3]{AGS08} would apply directly. For this reason, it is useful to consider the largest metric $\hat d$ which is smaller than $d$. It is easy to see that $\hat d$ is induced by the crystalline norm $\|\cdot\|$ on $\R^2$ defined in \eqref{for:defn:cryst:norm}. We further set
\begin{equation} \label{for:defn:Dhat}
\hat d (\bx, \by) 
:=
\| \bx - \by \|,
\qquad
\hat D^2 (\bX, \bY)
:=
\sum_{i=1}^n \hat d^2 (\bx_i, \by_i),
\end{equation}
We define the $\hat D$-MMS, $\hat \Phi$, and $\hat \bZ_\tau^k$ analogously to \eqref{for:MMS:D}-\eqref{for:defn:Phi:E} by replacing $D$ with~$\hat D$. \smallskip\\
\textbf{EDI}: Now we introduce the energy-dissipation inequality (EDI). To this aim, we first show that \eqref{for:evo:screw:BFLM} can be written as a (generalised) gradient flow (cf.~\cite{Mielke14LN}). 

\begin{definition}[Generalised gradient flow] \label{defn:genzd:GF}
A triple $(\cX, \mathsf E, \cR)$ is called a generalised gradient flow if
\begin{itemize}
  \item the state space $\cX$ is a smooth manifold;
  \item the energy functional $\mathsf E : \cX \to (-\infty, \infty]$ is smooth enough for the subdifferential $\operatorname D \mathsf E (x) \in T^\ast_x \cX$ to be well-defined for all $x \in \Dom \mathsf E$;
  \item the dissipation potential $\cR : T \cX \to [0,\infty]$ is convex, lower semicontinuous, and satisfies $\cR (x, 0) = 0$ for all $x \in \Dom \mathsf E$.
\end{itemize}
The related evolution is given by
\begin{equation} \label{for:defn:rate:eqn:genl}
\dot x \in \operatorname D_\xi \cR^\ast \bighaa{ x, - \operatorname D \mathsf E (x) },
\quad \text{in } \cX \text{ for a.e.~} t \in (0, T),
\end{equation}
where $\cR^\ast (x, \cdot )$ is the Legendre transform of $\cR (x, \cdot )$, and $\operatorname D_\xi$ denotes the subdifferential with respect to the second argument of $\cR^\ast$.

\end{definition}

We consider the generalised gradient flow given by the triple $( (\R^2)^n, E, \Psi )$, where $\Psi$ is defined by its Legendre transform
\begin{equation} \label{for:defn:Psi:ast:screw}
\begin{aligned}
  \psi^\ast &: \R^2 \to [0, \infty), 
  &\psi^\ast(\xi) 
  &:= \frac12 \| \xi \|_\ast^2 
  , \\
  \Psi^\ast &: (\R^2)^n \to [0, \infty), 
  &\Psi^\ast(\bxi) 
  &:= \sum_{i=1}^n \psi^\ast (\xi_i).
\end{aligned}
\end{equation}
With this choice, we obtain \eqref{for:evo:screw:BFLM} as the evolution \eqref{for:defn:rate:eqn:genl} of the triple $( (\R^2)^n, E, \Psi )$. This is easy to see after computing
\begin{align*}
  \operatorname D \Psi^\ast (\bxi)
  &= \operatorname D \psi^\ast (\xi_1) \times \ldots \times \operatorname D \psi^\ast (\xi_n)
  \subset (\R^2)^n, \\
  \operatorname D \psi^\ast (\xi) 
  &= \left\{ \begin{aligned}
    &0,
    &&\text{if } \xi = 0, \\
    &\bigacc{ (\bg \cdot \xi ) \bg },
    &&\text{if } \bg \in \cG \text{ the unique maximizer of } \bg \cdot \xi, \\
    &\operatorname{co} \bigacc{ (\bg' \cdot \xi ) \bg', (\bg'' \cdot \xi ) \bg'' },
    &&\text{if } \bigacc{ \bg', \bg'' } = \argmax_{\bg \in \cG} \bg \cdot \xi.
  \end{aligned} \right.
\end{align*}

By the basic properties of the Legendre transform, we obtain
\begin{equation*}
  \psi (\dot \bx) = \frac12 \| \dot \bx \|^2,
  \qquad \Psi (\dot \bX) = \sum_{i=1}^n \psi (\dot \bx_i).
\end{equation*}
We remark that by the Euclidean structure of $(\R^2)^n$, we can identify the tangent space and its dual at any $\bX \in (\R^2)^n$ with $(\R^2)^n$. On the other hand, we do distinct in our notation between particle positions $\bX$, velocities $\dot \bX$, and forces $\bxi$, because these objects have a different interpretation. Likewise, we use the metric $\hat D$ to measure the distance between two particle configurations, $\|\cdot\|$ to measure velocities, and $\|\cdot\|_\ast$ to measure forces.

The evolution can be written equivalently as a force balance or as a power balance. An overview of these descriptions is given in \cite{Mielke14LN}. By using that $E \in C^1 ((\R^2)^n)$, \cite[Theorem~3.2]{Mielke14LN} provides a fourth equivalent description of \eqref{for:defn:rate:eqn:genl}, called the EDI. For $( (\R^2)^n, E, \Psi )$, the EDI reads
\begin{equation} \label{for:EDIhat}
E(\bZ(T)) - E(\bZ(0)) 
+ \intabx0T{ \Psi \lrhaa{ \dot \bZ(t) } + \Psi^\ast \bighaa{ - \nabla E(\bZ(t)) } }t 
\leq 0.
\end{equation}

To define a solution concept for curves $\bZ$ satisfying \eqref{for:EDIhat}, we define the space of absolutely continuous curves \cite[Def.~1.1.1]{AGS08}. For $p \in [1,\infty]$, we say that $\bX \in AC^p(0,T;(\R^2)^n)$ if there exists an $f \in L^p(0,T)$ such that
\begin{equation} \label{eqn:defn:AC}
  \bigabs{ \bX(s) - \bX(t) }_2 \leq \intabx st{f(r)}r,
  \qquad \text{for all } 0 < s \leq t < T.
\end{equation}
By \cite[Rem.~1.1.3]{AGS08}, we have for $\bX \in AC^p(0,T;(\R^2)^n)$ that its derivative is defined almost everywhere.

\begin{definition}[Solution to EDI] 
A curve $\bZ : (0,T) \to (\R^2)^n$ is a solution to the EDI if $\bZ \in AC^2(0,T;(\R^2)^n)$ satisfies \eqref{for:EDIhat}.

\end{definition}

The EDI can also be written in terms of the \emph{right metric derivative} of $\bZ (t)$ and \emph{metric slope} of $E (\bZ (t))$ (for the precise definition and basic properties, see e.g.~\cite[Chapter~1]{AGS08}). For our purposes it suffices to define them, with respect to the metric $\hat D$, respectively as
\begin{equation} \label{for:defn:metric:der:and:slope}
  |\bZ'|_{\hat D} (t) 
  := \lim_{s \downarrow t} \frac{ \hat D \bighaa{ \bZ (s), \bZ (t) } }{s-t},
  \quad \text{and} \quad
  |\partial E|_{\hat D} (\bZ) 
  := \limsup_{\bX \to \bZ} \frac{E(\bZ) - E(\bX)}{ \hat D \haa{ \bX, \bZ } }.
\end{equation}
The right metric derivative and the metric slope with respect to the `quasi-distance' $D$ are defined analogously.

Lemma \ref{lem:metric:der:and:Psi} and Lemma \ref{lem:eqy:slopes} below contain the standard relation between $\Psi$--$\Psi^\ast$ and the metric derivatives and slopes. Together they imply that the EDI \eqref{for:EDIhat} is equivalent to
\begin{equation} \label{for:EDI:ito:metric:stuff}
E(\bZ(T)) - E(\bZ(0)) 
+ \frac12 \intabx0T{ |\bZ'|^2_{\hat D} (t) + |\partial E|^2_{\hat D} (\bZ (t)) }t 
\leq 0.
\end{equation}
In addition, Lemma \ref{lem:eqy:slopes} guarantees equality of the slopes with respect to $D$ and~$\hat D$.

\subsection{Basic estimates}
\label{chap:screw:dyncs:ssec:prelim}

In the following two lemmas we prove the relation between the couple $\Psi$--$\Psi^*$ and the metric derivatives and slopes when related to $D$ and $\hat D$.
\begin{lemma} [Relation between metric derivative and $\Psi$] \label{lem:metric:der:and:Psi} 
For any $\bZ \in AC^2 \bighaa{ [0,T]; (\R^2)^n }$, it holds
\begin{equation*}
  \frac12 |\bZ'|_{\hat D}^2 (t) = \Psi \lrhaa{ \dot \bZ (t) },
  \qquad \text{for a.e.~} t \in (0, T). 
\end{equation*}
\end{lemma}

\begin{proof}
Since $\hat d$ is induced by the norm $\| \cdot \|$, it holds for a.e.~$t \in (0, T)$ that
\begin{equation*}
  |\bZ'|_{\hat D}^2 (t) 
  = \lim_{s \downarrow t} \sum_{i=1}^n \frac{ \bignorm{ \bz_i (s) - \bz_i (t) }{}^2 }{ (s-t)^2 }
  = \sum_{i=1}^n \bignorm{ \dot \bz_i (t) }{}^2  
  = 2 \Psi \lrhaa{ \dot \bZ (t) }.\qedhere
\end{equation*}
\end{proof}

\begin{lemma}[Equality of slopes]  \label{lem:eqy:slopes}
For $E \in C^1 ((\R^2)^n)$, it holds
\begin{equation*}
  \frac12 |\partial E|_D^2 
  = \frac12 |\partial E|_{\hat D}^2 
  = \Psi^\ast ( \nabla E ).
\end{equation*}
\end{lemma}

\begin{proof}
The second equality is given by \cite[Corollary~1.4.5]{AGS08}. For the first equality, we have $|\partial E|_D \leq |\partial E|_{\hat D}$ by the definition of $D$ and $\hat D$. For the opposite inequality, we observe
\begin{align} \label{lem:eqt:slopes:pf:1}
    |\partial E|_D (\bX) 
    = \limsup_{\bY \to \bX} \frac{\nabla E(\bX) \cdot (\bX - \bY) + o \big( |\bX - \bY|_2 \big)}{D (\bX,\bY)}.
\end{align}
Next we are going to construct a particular sequence $\bY_\varepsilon \to \bX$ as $\varepsilon \to 0$. Let
\begin{equation*}
  \bar \bg_i \in \argmax_{\bg\in\cG} \left[ \bg \cdot \nabla_i E (\bX) \right],
  \quad \by_i^\varepsilon := \bx_i + \varepsilon \left[ \bar \bg_i \cdot \nabla_i E (\bX) \right] \bar \bg_i
  \qquad i = 1,\ldots,n,
\end{equation*}
where $\nabla_i E (\bX) := \nabla_{\bx_i} E (\bX) \in \R^2$. We note that, by definition of $\|\cdot\|_\ast$, 
\begin{equation*}
  \bar \bg_i \cdot \nabla_i E (\bX) = \| \nabla_i E (\bX) \|_\ast =: u_i.
\end{equation*}
Using the explicit sequence $(\bY_\varepsilon)$ in \eqref{lem:eqt:slopes:pf:1}, we obtain that
\begin{align*} 
    |\partial E|_D (\bX) 
    &\geq \limsup_{\varepsilon \to 0} 
         \frac{ 
           \sum_{i=1}^n \bigbhaa{ \varepsilon u_i \bar \bg_i \cdot \nabla_i E(\bX) + o ( \varepsilon u_i ) }
           }{ 
           \varepsilon \bighaa{ \sum_{i=1}^n u_i^2 }^{1/2}
         } \\
    &= \limsup_{\varepsilon \to 0} 
         \frac{ 
           \varepsilon | \bu |_2^2 + o \big( \varepsilon | \bu |_2 \big)
           }{ 
           \varepsilon | \bu |_2
         }
    = |\bu|_2 = \sqrt{ 2 \Psi^\ast ( \nabla E (\bX) ) }.\qedhere
\end{align*}
\end{proof}

Now we prove two estimates that are crucial for the proof of Theorem~\ref{thm:chap:screw:dyncs}.

\begin{lemma}[Estimate on single time step MMS] \label{lem:est:MMS}
Given an energy $E$ such that $E \in C^1 ((\R^2)^n) \cap W^{1,\infty} ((\R^2)^n)$, there exists $C > 0$ such that for any $\tau > 0$, $\bX \in (\R^2)^n$ we have
\begin{equation} \label{for:lem:est:MMS}
  D(\bX, \bY) \leq C \tau,
  \quad \text{and} \quad \hat D(\bX, \hat \bY) \leq C \tau,
\end{equation}
where $\bY, \hat \bY \in (\R^2)^n$ are a minimisers of respectively $\Phi(\bX, \cdot, \tau)$ and $\hat \Phi(\bX, \cdot, \tau)$. 
\end{lemma}

\begin{proof}
We have
\begin{align*} 
    E \left( \bX \right)
    \geq \inf_{\bZ \in (\R^2)^n} \left[ \frac{ D^2 \left( \bX, \bZ \right) }{2\tau} + E (\bZ) \right]
    = \frac{ D^2 \left( \bX, \bY \right) }{2\tau} + E \left( \bY \right)
\end{align*}
and hence, by Lipschitz continuity of $E$, we obtain \eqref{for:lem:est:MMS} from 
\begin{align*} 
    \frac{ D^2 \left( \bX, \bY \right) }{2\tau}
    \leq E \left( \bX \right) - E \left( \bY \right)
    \leq C D \left( \bX, \bY \right).
\end{align*}

Since the proof for $\hat D$ is analogous, we omit it.
\end{proof}

\begin{lemma}[Bound on directional slope] \label{lem:dir:slope}
Let $E \in C^1 ((\R^2)^n)$ bounded from below, $\bX \in (\R^2)^n$, $\tau > 0$, and $\bY \in \argmin \Phi(\bX, \cdot ,\tau)$. Then
\[
\frac{ D ( \bX, \bY ) }{\tau} 
\geq \sup \lraccvl{ 
\lim_{\bY_\varepsilon \to \bY} \frac{ E(\bY) - E(\bY_\varepsilon) }{ D( \bY, \bY_\varepsilon ) } 
}{ 
(\by^\varepsilon_i - \bx_i) \parallelsum (\by_i - \bx_i) \text{ for all } i = 1,\ldots,n
}.
\]

\end{lemma}

\begin{remark} \label{rem:main:diff:AGS}
Lemma \ref{lem:dir:slope} follows directly from the proof of Lemma~3.1.3 from \cite{AGS08}. We interpret the right-hand side of the inequality in Lemma \ref{lem:dir:slope} as a `directional slope'. It is clear from the definition of the slope \eqref{for:defn:metric:der:and:slope} that the directional slope is smaller than or equal to the `standard' slope $|\partial E|_D (\bY)$.

In fact, \cite[Lemma~3.1.3]{AGS08} states that if $D$ were a metric (i.e.~satisfies the triangle inequality), then a stronger inequality than the one in Lemma \ref{lem:dir:slope} would hold, in which the right-hand side is replaced by $|\partial E|_D (\bY)$. Such a stronger inequality is required in \cite{AGS08} to prove convergence of the MMS to the EDI.

Since $D$ is not a metric, we cannot apply this stronger inequality. It is not hard to construct an example in which $|\partial E|_D (\bY)$ is indeed larger than the `directional' slope. The bulk of our proof of the main result (Theorem \ref{thm:chap:screw:dyncs}) concerns an alternative argument to \cite[Lemma~3.1.3]{AGS08}.
\end{remark}

\section{Convergence of the $D$-MMS to the EDI}
\label{chap:screw:dyncs:sec:thm:and:pf}

We now prove the main result of the paper.
\begin{theorem}\label{thm:chap:screw:dyncs}
Let $T > 0$ be an end time, $\bZ^0 \in (\R^2)^n$ an initial condition, and $E \in C^1 ((\R^2)^n)$ be an energy such that $E$ and $\nabla E$ are bounded and uniformly continuous. For any time step $\tau > 0$, let $(\bZ_\tau^k)_{k \in \N{}}$ be a solution to the $D$-MMS \eqref{for:MMS:D} and $\overline \bZ_\tau : [0, T] \to (\R^2)^n$ the corresponding step function defined in \eqref{for:defn:Zbar}. Then, along a subsequence of $\tau \to 0$, the curves $\overline \bZ_\tau(t)$ converge pointwise for all $t \in [0, T]$ to an absolutely continuous curve $\hat \bZ : [0, T] \to (\R^2)^n$ which satisfies the EDI \eqref{for:EDI:ito:metric:stuff}.
\end{theorem}

\begin{proof} We split the proof in four steps. In Step~1 we use a refined version of Ascoli-Arzel\`a to identify the limiting curve $\tilde \bZ \in AC^2 ([0, T]; (\R^2)^n)$ for a subsequence of $\overline \bZ_\tau$ as $\tau \to 0$. Step~2 is a consequence of Step~1, in which we precisely state in what sense the interpolants $\overline \bZ_\tau$ and $\bZ_\tau^\Gamma$ \eqref{for:defn:De:Giorgi:itplt} converge. Step~3 is the main novelty of this paper, in which we provide an alternative argument to \cite[Lemma~3.1.3]{AGS08} which allows us in Step~4 to pass to the limit as $\tau \to 0$ to obtain that term in the EDI which is related to the slope $|\partial E|_{\hat D}$. In Remark \ref{rem:main:diff:AGS} we discuss the relation between \cite[Lemma~3.1.3]{AGS08} and the weaker result in Lemma \ref{lem:dir:slope}.

\smallskip
\emph{Step 1: Compactness of the step functions $\overline \bZ_\tau$.} \\ 
We prove that $\overline \bZ_\tau$ converges point-wise in time along a subsequence to some 
$$
\tilde \bZ \in AC^2 \big( [0, T]; (\R^2)^n \big).
$$
\cite[Proposition~3.3.1]{AGS08} guarantees the existence of $\tilde \bZ \in C \big( [0, T]; (\R^2)^n \big)$ provided that there exists a symmetric function $\omega \in [0, T]^2 \to [0, \infty)$ satisfying
\begin{equation} \label{for:omega:recs}
  \limsup_{\tau \to 0} \hat D \left( \overline \bZ_\tau (s), \overline \bZ_\tau (t) \right)
  \leq \omega(s,t) 
  \xrightarrow{|s-t| \to 0} 0.
\end{equation}
We prove \eqref{for:omega:recs} by estimating $\hat D \left( \overline \bZ_\tau (s), \overline \bZ_\tau (t) \right)$ for arbitrary $s,t \in [0,T]$, where $s \leq t$ without loss of generality. Let $K = \lceil s/\tau \rceil$ and $L = \lceil t/\tau \rceil$, and note that $s \in ( (K-1) \tau, K \tau ]$, $t \in ( (L-1) \tau, L \tau ]$. Then by the definition of $\overline \bZ_\tau$ \eqref{for:defn:Zbar}, we have $\overline \bZ_\tau (s) = \bZ_\tau^K$ and $\overline \bZ_\tau (t) = \bZ_\tau^L$, and we estimate
\begin{align*} 
   \hat D \left( \overline \bZ_\tau (s), \overline \bZ_\tau (t) \right)
   = \hat D \left( \bZ_\tau^K, \bZ_\tau^L \right)
   \leq \sum_{k = K}^{L-1} \hat D \left( \bZ_\tau^k, \bZ_\tau^{k+1} \right).
\end{align*}
We continue the estimate by using Lemma \ref{lem:est:MMS} to obtain
\begin{align} \nonumber
     \hat D \left( \overline \bZ_\tau (s), \overline \bZ_\tau (t) \right)
     &\leq \sum_{k = K}^{L-1} C \tau 
     = C (L-K) \tau
     \leq C \tau \big( \lceil t/\tau \rceil - \lceil s/\tau \rceil \big) \\\label{eqn:est:Zbar:st}
     &\leq C \tau \left( \frac t\tau - \frac s\tau + 1 \right)
     = C (t - s) + C \tau,
\end{align}
from which we conclude that \eqref{for:omega:recs} is satisfied for $\omega(s, t) = C |t-s|$. In the sequel, we proceed with the subsequence provided by \cite[Proposition~3.3.1]{AGS08} without changing notation. 

We prove that $\tilde \bZ$ is absolutely continuous by showing that \eqref{eqn:defn:AC} holds for some constant function $f$. From $\overline \bZ_\tau (t) \to \tilde \bZ (t)$ for a.e.~$t \in (0, T)$ and \eqref{eqn:est:Zbar:st}, we deduce that
\begin{equation} \label{eqn:Ztilde:AC}
  \hat D \lrhaa{ \tilde \bZ (s), \tilde \bZ (t) }
  = \lim_{\tau \to 0} \hat D \lrhaa{ \overline \bZ_\tau (s), \overline \bZ_\tau (t) }
  \leq \int_s^t C,
  \qquad \text{for a.e.~} 0 < s < t < T.
\end{equation}
Since $\tilde \bZ$ is continuous, we conclude that \eqref{eqn:Ztilde:AC} holds for all $0 \leq s \leq t \leq T$.

\smallskip
\emph{Step 2: Convergence of interpolants $\overline \bZ_\tau$ and $\bZ_\tau^\Gamma$ \eqref{for:defn:De:Giorgi:itplt}.} \\
We fix $t \in [0, T]$, and let $k_\tau = \ceil{t/\tau} - 1$ as in \eqref{for:def:itau}. The compactness result in Step~1 implies directly that $\hat D (\bZ_{k_\tau + 1}^\tau, \tilde \bZ (t)) \to 0$ pointwise in $t$ as $\tau \to 0$. Then from Lemma \ref{lem:est:MMS} and \eqref{for:est:on:ZGamma} we also have 
\begin{equation*}
  \hat D \bighaa{ \bZ^{k_\tau}_\tau, \tilde \bZ (t)}
  + \hat D \bighaa{ \bZ_\tau^\Gamma (t), \tilde \bZ (t)} 
  \xto{ \tau \to 0 } 0.
\end{equation*}
As a result of this and $\nabla E$ being continuous, we have that $\nabla E$ evaluated at $\bZ^{k_\tau}_\tau$, $\bZ^{k_\tau + 1}_\tau$ and $\bZ_\tau^\Gamma (t)$ converge to $\nabla E (\tilde \bZ (t))$ as $\tau \to 0$.
 
\smallskip
\emph{Step 3: Convergence of directional derivative to slope.}  \\
We fix an arbitrary $t \in [0,T]$ and set $k_\tau$ as in \eqref{for:def:itau}. Let 
\begin{equation} \label{eqn:defn:GGamma}
  \bG^\Gamma = \lracc{ \bg^\Gamma_1, \ldots, \bg^\Gamma_n } \in \cG^n 
  \quad \text{be such that} \quad 
  \bg_i^\Gamma \parallelsum \bigbhaa{ \bZ_\tau^\Gamma (t) - \bZ^{k_\tau}_\tau }_i \text{ for all } i = 1,\ldots,n.
\end{equation} 
We prove in this step that
\begin{equation} \label{thm:MMS:to:EDI:pf:grad:conv}
  \sum_{i=1}^n \left[ \bg^\Gamma_i \cdot \nabla_i E (\bZ_\tau^\Gamma (t)) \right]^2
  \xrightarrow{ \tau \to 0 }
  \abs{ \partial E }_{\hat D}^2 (\tilde \bZ (t)).
\end{equation}
From the proof of Lemma \ref{lem:eqy:slopes} and the definition of $\Psi^\ast$ \eqref{for:defn:Psi:ast:screw}, it follows that \eqref{thm:MMS:to:EDI:pf:grad:conv} is implied by the claim
\begin{equation} \label{thm:MMS:to:EDI:pf:grad:conv:i}
  \left[ \bg^\Gamma_i \cdot \nabla_i E (\bZ_\tau^\Gamma (t)) \right]^2
  \xrightarrow{ \tau \to 0 }
  \max_{\bg \in \cG} \bigbhaa{ \bg \cdot \nabla_i E (\tilde \bZ (t)) }^2,
  \qquad \text{for all } i = 1,\ldots, n.
\end{equation}
For proving \eqref{thm:MMS:to:EDI:pf:grad:conv:i}, we set $\ba_i := \nabla_i E (\tilde \bZ (t))$ and $\overline \bg_i$ as a minimiser of $-\bg \cdot \ba_i$. If $\ba_i = \bzero$, then \eqref{thm:MMS:to:EDI:pf:grad:conv:i} follows directly from Step~2.

The main part of the proof of \eqref{thm:MMS:to:EDI:pf:grad:conv:i} for $\ba_i \neq \bzero$ is to characterise $\bG^\Gamma$. We recall from \eqref{for:defn:De:Giorgi:itplt} that
\begin{equation} \label{thm:MMS:to:EDI:pf:DGitplt:1}
  \bZ^\Gamma_\tau (t) \in \argmin_{\bX \in (\R^2)^n} \lrbhaa{
    \frac1{2 \delta} D^2 (\bZ_\tau^{k_\tau}, \bX) + E(\bX)  
  },
\end{equation}
where $\delta = t - \tau k_\tau \in (0, \tau]$. Since $D$ is only finite on $\R{} \cG$, we can restrict the minimisation over $\bX = (\bx_1, \ldots, \bx_n)^T$ to $\bx_i = \bz_{i, \tau}^{k_\tau} + \alpha_i \bg_i$ with $\alpha_i \geq 0$ and $\bg_i \in \cG$. Thanks to Lemma \ref{lem:est:MMS} we can assume that $\alpha_i \leq C \delta$. At the same time, we expand $E(\bX)$ in a Taylor series around $\bZ_\tau^{k_\tau}$. By these arguments, we rewrite the minimisation problem in \eqref{thm:MMS:to:EDI:pf:DGitplt:1} as
\begin{equation} \label{thm:MMS:to:EDI:pf:DGitplt:2}
  E(\bZ_\tau^{k_\tau}) + \min_{\bg_i \in \cG} \min_{\alpha_i \geq 0} \sum_{i=1}^n \lrbhaa{
    \frac1{2 \delta} \alpha_i^2 + \alpha_i \bg_i \cdot \nabla_i E (\bZ_\tau^{k_\tau}) + o (\delta)
  }.
\end{equation}
We recall from \eqref{thm:MMS:to:EDI:pf:DGitplt:1} that $\bg_i^\Gamma$ is a minimiser of this minimisation problem. We characterise this minimiser by solving \eqref{thm:MMS:to:EDI:pf:DGitplt:1} first under the assumption that $o (\delta) = 0$. Minimising \eqref{thm:MMS:to:EDI:pf:DGitplt:2} over $\alpha_i \geq 0$ yields
\begin{equation} \label{thm:MMS:to:EDI:pf:DGitplt:3}
  E(\bZ_\tau^{k_\tau}) - \frac\delta2 \sum_{i=1}^n 
    \Bighaa{ \min_{\bg_i \in \cG} \, \bg_i \cdot \nabla_i E (\bZ_\tau^{k_\tau}) }^2.
\end{equation}
We treat the minimisation within parentheses for each $i = 1, \ldots, n$ separately. For convenience we drop the index $i$ in our notation whenever possible. We split characterizing $\bg^\Gamma$ in two cases: (i) $-\bg \cdot \ba$ has a unique minimiser $\overline \bg$, and (ii) $-\bg \cdot \ba$ has exactly two minimiser $\overline \bg^1$ and $\overline \bg^2$.

In case (i), it holds that $\ba$ is an element of the cone $\Lambda_{\overline \bg}$ \eqref{for:defn:Lambda}. Hence, there exists an $r > 0$ such that $B(\ba, r) \subset \subset \Lambda_{\overline \bg}$. Then, by Step~2, it holds that $\nabla_i E (\bZ_\tau^{k_\tau}) \in B(\ba, r)$ for all $\tau$ small enough. Hence, $\overline \bg$ is the unique minimiser of $\bg \cdot \nabla_i E (\bZ_\tau^{k_\tau})$, and, furthermore, there exist a $c > 0$ independent of $\tau$ such that
\begin{equation*}
  \Bighaa{ \min_{\bg \in \cG \setminus \acc{ \overline \bg } } \, (\bg  - \overline \bg) \cdot \nabla_i E (\bZ_\tau^{k_\tau}) } \geq c > 0.
\end{equation*}
Hence, the minimiser of $\bg \cdot \nabla_i E (\bZ_\tau^{k_\tau})$ is stable under perturbations of the form $o (1)$, and thus $\overline \bg_i$ minimises the term in square brackets in \eqref{thm:MMS:to:EDI:pf:DGitplt:2}. Hence, $\bg^\Gamma_i = \overline \bg_i$, from which we conclude that \eqref{thm:MMS:to:EDI:pf:grad:conv:i} holds in case (i).

In case (ii), it holds that $\ba$ is an element of the cone $\Lambda$ spanned by $\overline \bg^1$ and $\overline \bg^2$. An analogous argument as used in case (i) implies that $\bg \cdot \nabla_i E (\bZ_\tau^{k_\tau})$ is minimised by $\overline \bg^1$ or $\overline \bg^2$, and, furthermore, there exist a $c > 0$ independent of $\tau$ such that
\begin{equation*}
  \Bighaa{ \min_{\bg \in \cG \setminus \acc{ \overline \bg^1, \overline \bg^2 } } \, (\bg  - \overline \bg^j) \cdot \nabla_i E (\bZ_\tau^{k_\tau}) } \geq c > 0,
  \qquad \text{for } j = 1,2.
\end{equation*}
Hence, either $\overline \bg^1$ or $\overline \bg^2$ minimises the term in square brackets in \eqref{thm:MMS:to:EDI:pf:DGitplt:1}. This implies $\bg^\Gamma \in \acc{ \overline \bg^1, \overline \bg^2 }$, and \eqref{thm:MMS:to:EDI:pf:grad:conv:i} follows from the observation that
\begin{equation*}
  \overline \bg^j \cdot \nabla_i E (\bZ_\tau^\Gamma (t))
  \xrightarrow{ \tau \to 0 }
  \overline \bg^j \cdot \ba
  = \min_{\bg \in \cG} \bg \cdot \ba,
  \qquad \text{for } j = 1,2.
\end{equation*}

\smallskip
\emph{Step 4: $\tilde \bZ$ satisfies the EDI.} \\
To simplify the proof, we assume that $\bZ_\tau^k$ and $\bZ^\Gamma_\tau$ are uniquely defined by \eqref{for:MMS:D} and \eqref{for:defn:De:Giorgi:itplt}. In the general case we should take into account the supremum and the infimum of $D^2(\bZ^\Gamma_\tau(t), \bZ_\tau^{k_\tau})$ over all possible instances of $\bZ^\Gamma_\tau(t)$. This is done in \cite[Chapter~3]{AGS08}, in which all equalities below become inequalities. 
 
Following the lines of \cite[Theorem~3.1.4]{AGS08}, we obtain for a fixed time step $\tau > 0$ that
\begin{equation*}
\frac{ D^2 \lrhaa{ \bZ_\tau^{k}, \bZ_\tau^{k+1} } }{2\tau} 
+ \intabx{\tau k}{\tau (k + 1)}{ \frac{D^2 \bighaa{ \bZ^\Gamma_\tau(t), \bZ_\tau^{k} } }{2(t - \tau k)^2} }t 
= E(\bZ_\tau^{k}) - E(\bZ_\tau^{k+1}), 
\qquad\quad\text{for } k = 0, \ldots, \ceil{T / \tau} - 1.
\end{equation*}
Summation over $k$ results in
\begin{equation} \label{thm:MMS:to:EDI:pf:disc:EDI}
\sum_{k=0}^{\lceil T/\tau \rceil - 1 }\frac{D^2(\bZ_\tau^k,\bZ_\tau^{k+1})}{2\tau} 
+ \intabx0{\lceil T/\tau \rceil \tau }{ \frac{D^2 \bighaa{ \bZ^\Gamma_\tau(t), \bZ_\tau^{ \floor{ t/\tau } } } }{ 2 (t - \floor{t / \tau} \tau )^2 } }t 
= E(\bZ^0) -  E \bighaa{ \bZ_\tau^{\lceil T/\tau \rceil} }.
\end{equation}

Next we show how to pass to the limit in the terms of \eqref{thm:MMS:to:EDI:pf:disc:EDI} as $\tau \to 0$. We set $k_\tau (t) = \ceil{t/\tau} - 1$ as in \eqref{for:def:itau}. By 
Step~2, the right-hand side of \eqref{thm:MMS:to:EDI:pf:disc:EDI} converges to the difference between the energy evaluated at $\tilde \bZ (T)$ and $\bZ^0$. Regarding the first term in the left-hand side of \eqref{thm:MMS:to:EDI:pf:disc:EDI}, we use the estimate \cite[(3.3.10)]{AGS08} to obtain that
\begin{align*}
  \liminf_{\tau \to 0} \tau \sum_{k=0}^{k_\tau(T)} \frac{ D^2(\bZ_\tau^k,\bZ_\tau^{k+1}) }{ 2\tau^2 } 
&= \liminf_{\tau \to 0} \intabx0{ \ceil{T/\tau} \tau }{ \frac{ D^2 \bighaa{ \bZ_\tau^{k_\tau(t)}, \bZ_\tau^{k_\tau(t)+1} } }{ 2\tau^2 } }t  
\geq \frac12 \intabx0T{ \lrabs{ \tilde \bZ' }_{\hat D}^2(t) }t.  
\end{align*}
We prove the convergence of the second term in the left-hand side of \eqref{thm:MMS:to:EDI:pf:disc:EDI} as follows. We fix $t \in [0, T]$ and $\tau > 0$. We set $G^\Gamma$ as in \eqref{eqn:defn:GGamma} and $u_i := \bg_i^\Gamma \cdot \nabla_i E (\bZ_\tau^\Gamma (t))$. We obtain from Lemma \ref{lem:dir:slope} that
\begin{equation*}
  \frac{ D \bighaa{ \bZ^\Gamma_\tau(t), \bZ_ \tau^{k_\tau (t)} } }{ t - \tau k_\tau (t) }
  \geq \limsup_{\bepsilon \to 0} \frac1{|\bepsilon|_2} \sum_{i=1}^n \varepsilon_i \bg_i^\Gamma \cdot \nabla_i E (\bZ_\tau^\Gamma (t))
  = \limsup_{\bepsilon \to 0} \frac{ \bepsilon \cdot \bu }{|\bepsilon|_2}
  = |\bu|_2.
\end{equation*}
In Step~3 we prove that $|\bu|_2^2 \to \abs{ \partial E }_{\hat D}^2 (\tilde \bZ (t))$ as $\tau \to 0$. Hence, by applying Fatou's Lemma, we obtain
\begin{align*}
\liminf_{\tau \to 0} \frac12 \intabx0T{ \frac{ D^2 \bighaa{ \bZ^\Gamma_\tau(t), \bZ_ \tau^{\floor{ t/\tau } } } }{ (t - \floor{t / \tau} \tau )^2 } }t 
&\geq \frac12 \intabx0T{ |\partial E |_{\hat D}^2 \bighaa{ \tilde \bZ(t) } }t.    
\end{align*}
Combining the results above, we obtain after passing to the limit $\tau \to 0$ in \eqref{thm:MMS:to:EDI:pf:disc:EDI}
\[
\frac12 \intabx0T{ | \tilde \bZ '|_{\hat D}^2(t) }t 
+ \frac12 \intabx0T{ |\partial E |_{\hat D}^2 \bighaa{ \tilde \bZ(t) } }t 
\leq E(\bZ^0) - E(\tilde \bZ(T)),
\]
from which we conclude by Lemma \ref{lem:metric:der:and:Psi} and Lemma \ref{lem:eqy:slopes} that $\tilde \bZ$ satisfies the EDI.
\end{proof}

\section{Generalisations of Theorem \ref{thm:chap:screw:dyncs}}
\label{sec:extn}

Here we discuss in more detail two of the three remarks on Theorem \ref{thm:chap:screw:dyncs} that are mentioned in Section \ref{sec:dislo_res_com}.

\subsection{Extension to higher dimensions}
\label{ssec:extn}

Theorem \ref{thm:chap:screw:dyncs} can be generalised to higher dimensions for the particle positions. We show this by considering $n$ particles with positions $\bz_i \in \R^d$ with $d \geq 2$. Since the setting is analogous to the two-dimensional scenario, we keep the same notation, and only mention the important differences.

The main difference in higher dimensions is that we consider the set of glide directions given by
\begin{equation*} \label{for:defn:cG:genzd}
  \cG := \acc{ \bg_1, \ldots, \bg_N } \subset \mathbb S^{d-1},
\end{equation*}
which satisfies the two properties 
\begin{equation*} \label{for:G:props:genzd}
  \bg \in \cG \rar -\bg \in \cG,
  \quad \text{and} \quad
  \markops{\cG} = \R^d. 
\end{equation*}
The related crystalline norm reads
\begin{equation} \label{for:defn:cryst:norm:genzd}
 \| \bx \| := \min \biggaccv{ 
 \sum_{k=1}^N \alpha_k 
 }{
  \alpha_k \geq 0 \text{ such that } \: \sum_{k=1}^N \alpha_k \bx_k = \bx
  }.
\end{equation}
Next we show that the dual norm can be characterised by
\begin{equation} \label{for:defn:dual:cryst:norm:genzd}
 \| \bx \|_\ast 
 := \max_{\by \in \R^d \setminus \acc 0} \frac{ \bx \cdot \by }{ \norm{\by}{} }
 = \max_{\bg \in \cG} \bg \cdot \bx.
\end{equation}
Since the maximum is taken over a smaller space in the right-hand side of \eqref{for:defn:dual:cryst:norm:genzd}, it suffices to prove that it is larger than or equal to the left-hand side. To this aim, let $\by$ be a maximiser of the left-hand side, and let $\alpha_k \geq 0$ as in \eqref{for:defn:cryst:norm:genzd}. We conclude by estimating
\begin{align*}
  \| \bx \|_\ast 
 = \frac{ \bx \cdot \by }{ \norm{\by}{} }
 = \bigghaa{ \sum_{k=1}^N \alpha_k }^{-1} \sum_{k=1}^N \alpha_k \bg_k \cdot \bx
 \leq \max_{\bg \in \cG} \bg \cdot \bx.
\end{align*}
In terms of the crystalline norm and its dual we define the distances $D$ and $\hat D$ and the related $\Psi$ and $\Psi^\ast$ analogously to the two-dimensional setting. With these objects, the MMS \eqref{for:MMS:D} and EDI \eqref{for:EDIhat} are defined analogously.

It is readily checked that the basic estimates in Section \ref{chap:screw:dyncs:ssec:prelim} and most steps in the proof of Theorem \ref{thm:chap:screw:dyncs} hold in the $d$-dimensional case by analogous arguments. The only part where the extension to $d$ requires a modification is for the characterisation of the minimiser of the term within parenthesis in \eqref{thm:MMS:to:EDI:pf:DGitplt:3} in Step 3 in the proof of Theorem \ref{thm:chap:screw:dyncs}. The remainder of this section describes this modification.

First, we define the cone $\Lambda_\bg$ as an extension of \eqref{for:defn:Lambda} by 
\begin{equation} \label{for:Lambda:prop:genzd}
  \Lambda_\bg
  := 
  \biggaccv{ 
\bx \in \R^d
}{
\bg \cdot \bx > \max_{\tilde \bg \in \cG \setminus \acc \bg} \tilde \bg \cdot \bx
}.
\end{equation}
We remark that $\acc{ \Lambda_\bg \cap \mathbb S^{d-1} }_{\bg \in \cG}$ describes the Voronoi tessellation of $S^{d-1}$ with respect to $\cG$.

With \eqref{for:Lambda:prop:genzd} it is easy to see that the characterisation of the minimiser in \eqref{thm:MMS:to:EDI:pf:DGitplt:3} can be done with an analogous argument when $\ba \in \Lambda_\bg$ for some $\bg \in \cG$. If $\ba \notin \cup_{\bg \in \cG} \Lambda_\bg$, then 
\begin{equation*}
  \Bigabs{ \argmax_{\bg \in \cG} \bg \cdot \ba } \geq 2,
\end{equation*}
while in the two-dimensional scenario the left-hand side equals $2$. In any case, the argument in Step~3 also holds when more than $2$ glide directions are considered.

\subsection{Annihilation of dislocations in a variational framework}
\label{chap:screw:dyncs:sec:NnC}

With Theorem~\ref{thm:chap:screw:dyncs} we proved that the time-continuous model in \cite{CermelliGurtin99} can be obtained as the limit of the time-discrete $D$-MMS schemes when the time step converges to $0$. We would like to address 
an open problem regarding annihilation of dislocations.

%

By regularizing the energy to prevent the minimising-movement scheme in \eqref{for:MMS:D:intro} to jump in the first time step to a state in which the energy equals $- \infty$, we remove the possibility for dislocations to annihilate or to be absorbed at $\partial \Omega$. This regularization of the energy is therefore artificial. Indeed, the evolution defined by \eqref{for:evo:screw:BFLM} can easily be modified to allow for annihilation; whenever two dislocations annihilate (or one dislocation leaves the domain $\Omega$), the evolution can be restarted by removing the annihilated dislocations from the equation, and taking the current positions of the other dislocations as the new initial condition. This raises the following question: is it possible to modify the MMS to describe such dynamics which allow for annihilation in a variational framework?

In \cite{AlicandroDeLucaGarroniPonsiglione14} the minimum in \eqref{for:MMS:D:intro} is taken over a $\tau$-independent neighbourhood $\cN$ of $\bZ_\tau^k$. With this strategy, the energy does not need to be regularized, but the MMS breaks down whenever the difference between any pair of opposite dislocations is in~$\cN$.

One idea to make this approach work, is to make $\cN$ dependent on $\tau$, such that it shrinks to a singleton as $\tau \to 0$, but at a slower rate than $\tau$ (e.g., $\sqrt \tau$). 

The main challenge is to treat the time step in which annihilation takes place. From the procedure above, it seems reasonable to restart the MMS with a new energy (consisting of fewer dislocations) while keeping fixed the positions of the dislocations which were not annihilated. We expect that passing to the limit $\tau \to 0$, if possible, requires serious modifications to the argument in \cite[Chapter~2 and 3]{AGS08}.

\bigskip
\noindent\textbf{Acknowledgements} The authors warmly thank the Centre for Analysis, Scientific computing and Applications (CASA) of the Eindhoven University of Technology, Eindhoven, The Netherlands and SISSA, Trieste, Italy, where this research was carried out.
G.A.B.\@ and kindly acknowledge support from the Nederlandse Organisatie voor Wetenschappelijk Onderzoek (NWO) VICI grant 639.033.008.
P.vM.\@ kindly acknowledges the financial support from the NWO Complexity grant 645.000.012.
The research of M.M.\@ was partially supported by the European Research Council through the ERC Advanced Grant ``QuaDynEvoPro'', grant agreement no.\@ 290888. M.M.\@ is a member of the Progetto di Ricerca GNAMPA-INdAM 2015 ``Fenomeni critici nella meccanica dei materiali: un approccio variazionale'' (INdAM-GNAMPA Project 2015 ``Critical phenomena in the mechanics of materials: a variational approach'').

\bibliography{My_bib_Bonaschi}
\bibliographystyle{alpha}

\end{document}